\documentclass[12pt]{article}
\usepackage{amsmath, amsthm, amssymb,amscd}

\title{Complexity of matrix problems\footnotetext{This is the authors' version of a work that was published in Linear Algebra Appl. 361 (2003) 203--222.}}
\author{Genrich R. Belitskii%
\thanks{Partially supported by
Israel Science Foundation, Grant
216/98.}\\ Department of Mathematics\\
Ben-Gurion University of the Negev\\
Beer-Sheva 84105, Israel\\
 genrich@indigo.cs.bgu.ac.il
\and
Vladimir V. Sergeichuk%
\thanks{The research was done while this
author was visiting the Ben-Gurion
University of the Negev supported by
Israel Science Foundation.}\\ Institute
of Mathematics\\ Tereshchenkivska 3,
Kiev, Ukraine\\sergeich@imath.kiev.ua}
\date{}

\begin{document}
\maketitle

\begin{abstract}
In representation theory, the
problem of classifying pairs of
matrices up to simultaneous
similarity is used as a measure
of complexity; classification
problems containing it are
called wild problems. We show in
an explicit form that this
problem contains all
classification matrix problems
given by quivers or posets. Then
we prove that it does not
contain (but is contained in)
the problem of classifying
three-valent tensors. Hence, all
wild classification problems
given by quivers or posets have
the same complexity; moreover, a
solution of any one of these
problems implies a solution of
each of the others. The problem
of classifying three-valent
tensors is more complicated.


 \end{abstract}

\newcommand{\rank}{\mathop{\rm rank}\nolimits}

\newtheorem{theorem}{Theorem}[section]
\newtheorem{lemma}{Lemma}[section]

%emlines.sty:
\def\newpic#1{%
 \def\emline##1##2##3##4##5##6{%
 \put(##1,##2){\special{em:point #1##3}}%
  \put(##4,##5){\special{em:point #1##6}}%
 \special{em:line #1##3,#1##6}}}
\newpic{}

\newcommand{\quiv}{%
\!\!
%TexCad Options
%\grade{\on}
%\emlines{\off}
%\beziermacro{\on}
%\reduce{\on}
%\snapping{\off}
%\quality{2.00}
%\graddiff{0.01}
%\snapasp{1}
%\zoom{5.00}
\unitlength 0.5mm \linethickness{0.4pt}
%\begin{picture}(9,0)(7,13.5)
\begin{picture}(9,0)(7,12.8)
\put(15.00,15.00){\circle*{0.7}}
\put(9.47,15.00){\oval(4.00,4.00)[l]}
\bezier{16}(9.53,16.93)(11.33,17.00)(13.73,15.93)
%\bezvec{20}(9.53,13.00)(11.80,13.13)(13.80,14.33)
\put(13.80,14.33){\vector(2,1){0.2}}
\bezier{20}(9.53,13.00)(11.80,13.13)(13.80,14.33)
%\end
\end{picture}
\! }

\newcommand{\pair}{%
\!
%TexCad Options
%\grade{\on}
%\emlines{\off}
%\beziermacro{\on}
%\reduce{\on}
%\snapping{\on}
%\quality{2.00}
%\graddiff{0.01}
%\snapasp{1}
%\zoom{5.00}
\unitlength 0.5mm \linethickness{0.4pt}
\begin{picture}(15,0)(-7,-2.2)
\put(0.00,0.00){\circle*{0.8}}
\put(-5.53,0.00){\oval(4.00,4.00)[l]}
\bezier{16}(-5.47,1.93)(-3.67,2.00)(-1.27,0.93)
%\bezvec{20}(-5.47,-2.00)(-3.20,-1.87)(-1.20,-0.67)
\put(-1.20,-0.67){\vector(2,1){0.2}}
\bezier{20}(-5.47,-2.00)(-3.20,-1.87)(-1.20,-0.67)
%\end
\put(5.53,0.00){\oval(4.00,4.00)[r]}
\bezier{16}(5.47,1.93)(3.67,2.00)(1.27,0.93)
%\bezvec{20}(5.47,-2.00)(3.20,-1.87)(1.20,-0.67)
\put(1.20,-0.67){\vector(-2,1){0.2}}
\bezier{20}(5.47,-2.00)(3.20,-1.87)(1.20,-0.67)
%\end
\end{picture}
\!\! }

\section{Introduction}   \label{s1}

Classification problems of
representation theory split into two
types: {\it tame} (or classifiable) and
{\it wild} (containing the problem of
classifying pairs of matrices up to
simultaneous similarity); wild problems
are hopeless in a certain sense. These
terms were introduced by Donovan and
Freislich \cite{don} in analogy with
the partition of animals into tame and
wild ones.

Gelfand and Ponomarev \cite{gel-pon}
proved that the problem of classifying
pairs of matrices up to simultaneous
similarity contains the problem of
classifying $t$-tuples of matrices up
to simultaneous similarity for an
arbitrary $t$. (The problem of
classifying pairs of linear operators
is as complicated as the problem of
classifying 1,000,000-tuples of linear
operators!) This implies that it
contains the problem of classifying
representations of an arbitrary
$t$-dimensional algebra\footnote{This
algebra is a factor algebra $\Lambda=
k\langle x_1,\dots,x_t\rangle/J$ of the
free algebra of noncommutative
polynomials in $x_1,\dots,x_t$. Let
$g_1,\dots,g_r$ be generators of $J$,
then each matrix representation of
$\Lambda$ is a $t$-tuple of $n\times n$
matrices $(A_1,\dots,A_t)$ satisfying
$g_i(A_1,\dots,A_t)=0,\ i=1,\dots,r$;
it determines up to simultaneous
similarity.}, whence it contains matrix
problems given by quivers.

In Section \ref{s2}, we give the proof
of the last statement by methods of
linear algebra; it was sketched in
\cite[Sect. 3.1]{ser}. The notions of a
quiver and its representations were
introduced by Gabriel \cite{gab} and
admit to formulate problems of
classifying systems of linear mappings
(without relations).

In Section \ref{s3}, we prove that the
problem of classifying pairs of
matrices up to simultaneous similarity
contains matrix problems given by
partially ordered sets. The notion of
poset representations was introduced by
Nazarova and Roiter \cite{naz_roi} and
admits to formulate problems of
classifying block matrices
$[A_1\,|\,A_2\,|\dots|\,A_t]$ up to
elementary row-transformations of the
whole matrix, elementary
column-transformations within each
vertical strip, and additions of a
column of $A_i$ to a column of $A_j$
for a certain set of pairs $(i,j)$.

In Section \ref{s4}, we prove that the
problem of classifying three-valent
tensors contains the problem of
classifying pairs of matrices up to
simultaneous similarity, but it is not
contained in the last problem.
Three-valent tensors are given by
spatial matrices, so we first consider
the problem of classifying $m\times
n\times q$ {\it spatial matrices}
${\mathbb A}
=[a_{ijk}]_{i=1}^m{}_{j=1}^n{}_{k=1}^q$
up to {\it equivalence
transformations}:
\begin{equation}\label{1.00}
[a_{ijk}]\longmapsto [a'_{ijk}],\quad
a'_{i'j'k'}=\sum_{ijk}
a_{ijk}r_{ii'}s_{jj'} t_{kk'},
\end{equation}
where
\begin{equation}\label{1.01}
 R=[r_{ii'}],\ S=[s_{jj'}],\
T=[t_{kk'}]
\end{equation}
are nonsingular $m\times m$, $n\times
n$, and $q\times q$ matrices. We
classify $m\times n\times 2$ spatial
matrices up to equivalence and prove
that the problem of classifying
$m\times n\times 3$ spatial matrices up
to equivalence contains (but is not
contained in) the problem of
classifying pairs of matrices up to
simultaneous similarity.
\smallskip

Every matrix problem $\cal A$ is given
by a set ${\cal A}_1$ of $a$-tuples of
matrices and a set ${\cal A}_2$ of
admissible transformations with them.
We say that a matrix problem ${\cal A}$
{\it is contained in} a matrix problem
${\cal B}$ if there exists a $b$-tuple
${\cal T}(x)={\cal T}(x_1,\dots,x_a)$
of matrices, whose entries are
noncommutative polynomials in
$x_1,\dots,x_a$, such that
\begin{itemize}
  \item[(i)]
${\cal T}(A)={\cal T}(A_1,\dots,A_a)
\in{\cal B}_1$ if $A=(A_1,\dots,A_a)\in
{\cal A}_1$,
  \item[(ii)]
for every $A,A'\in{\cal A}_1$, $A$
reduces to $A'$ by transformations
${\cal A}_2$ if and only if ${\cal
T}(A)$ reduces to ${\cal T}(A')$ by
transformations ${\cal B}_2$.
\end{itemize}

In this article (except for Theorem
\ref{t4.3}), the entries of matrices
from ${\cal T}(x)$ are 0, scalars, or
$x_i$, and we replace them by zero
matrices, scalar matrices, or $A_i$.
Suppose ${\cal A}$ is contained in
${\cal B}$ and a set of {\it canonical
$b$-tuples} for the problem ${\cal B}$
is known (this set must posses the
following property: each $b$-tuple
$A\in{\cal B}_1$ reduces to a canonical
$A_{\text{can}}\in{\cal B}_1$, and $A$
reduces to $B$ iff
$A_{\text{can}}=B_{\text{can}}$). We
reduce to the form ${\cal T}(A)$ those
canonical $b$-tuples, for which this is
possible. Then all $a$-tuples $A$ from
the obtained set of ${\cal T}(A)$ may
be considered as canonical $a$-tuples
for ${\cal A}$. Hence, a solution of
the problem ${\cal B}$ implies a
solution of ${\cal A}$.

In \cite{ser}, the entries of matrices
in the considered matrix problems
satisfied systems of linear equations,
for this reason the entries of matrices
from ${\cal T}(x)$ were linear
polynomials. In the theory of
representations of quivers with
relations, the entries of matrices from
${\cal T}(x)$ are noncommutative
polynomials.
\smallskip

A quiver or poset is called {\it tame}
({\it wild}) if the problem of
classifying its representations is tame
(wild). We sum up results of this
article in the following theorem:

\begin{theorem} \label{t1.2}
All problems of classifying
representations of wild quivers or
posets have the same complexity: any of
them contains every other (moreover, a
solution of one implies solutions of
the others). The problem of classifying
three-valent tensors is more
complicated since it contains each of
them but is not contained in them.
\end{theorem}

This theorem explains the existence of
the ``universal'' algorithm \cite{bel}
(see also \cite{bel1} or \cite{ser})
for reducing the matrices of an
arbitrary representation of a quiver or
poset to canonical form, and unsuccess
of the authors' attempts to extend to
three-valent tensors both this
algorithm and the method \cite{ser2}
for reducing the problem of classifying
systems of forms and linear mappings to
the problem of classifying linear
mappings (its analog would be a method
for reducing the classification of
systems of three-valent tensors to the
classification of spatial matrices up
to equivalence transformations).
\smallskip

The algorithm \cite{bel} was used in
\cite{ser_gal} in order to receive a
canonical form of $4\times 4$ matrices
up to simultaneous similarity. The
algorithm was also used in \cite{ser}
to prove that the set of canonical $m
\times n$ matrices for a {\it tame}
matrix problem forms a finite number of
points and straight lines in the affine
space of $m \times n$ matrices. This
statement is a strengthened form of
Drozd's Tame--Wild Theorem \cite{dro}
and holds for a large class of matrix
problems, which includes
representations of quivers and posets.
A full system of invariants for pairs
of matrices up to simultaneous
similarity was obtained by Friedland
\cite{fri}.

For each matrix problem, one has an
alternative: to solve it or to prove
that it is wild and hence is hopeless
in a certain sense. Examples of wild
problems:

(a) The problem of classifying pairs of
$m\times n$ and $n\times n$ matrices up
to transformations
$$
(A,B)\longmapsto (R^{-1}AR,SBR),
$$
where $R$ and $S$ are nonsingular
matrices (that is, the replacement of
the quiver \pair \, with
\quiv$\!\!\!\!\rightarrow$ \, does not
simplify the problem of classifying its
representations; see the list
\eqref{2.4a}).

(b) The problem of classifying pairs of
commuting nilpotent matrices $(A,B)$ up
to simultaneous similarity, see
\cite{gel-pon}; this problem was solved
in \cite{naz_bon} if $AB=BA=0$.

(c) The problem of classifying
quintuples of subspaces in a vector
space. A classification of quadruples
of subspaces (they may be given by
representations of the quiver
${}^{\searrow}_{\nearrow}\!\!\cdot\!\!
{}^{\swarrow}_{\nwarrow}$) was given in
\cite{gel-pon1}.

(d) The problem of classifying triples
of quadratic forms; its wildness
follows from the method of classifying
pairs of quadratic forms used in
\cite[Theorem 4]{ser2}. A
classification of all {\it tame}
systems of linear mappings, bilinear
forms, and quadratic forms (without
relations) was obtained in \cite[Sect.
4]{ser3}.

(e) The problem of classifying of
metric (or selfadjoint) operators in a
space with symmetric bilinear form; the
problem was solved by many authors if
this form is nonsingular, see
\cite[Theorems 5 and 6]{ser2}.

(f) The problem of classifying normal
operators in a space with indefinite
scalar product, see \cite{goh} or
\cite[Theorem 5.5]{ser3}.

In the theory of unitary matrix
problems, the role of pairs of matrices
up to simultaneous similarity is played
by the problem of classifying matrices
up to unitary similarity; it contains
the problem of classifying unitary
representations of an arbitrary quiver
(its points and arrows correspond to
unitary spaces and linear operators),
see \cite[Sect. 2.3]{ser1}.

The partition into tame and wild
problems was first exhibited for
representations of Abelian groups (see
\cite{gus}): Bashev \cite{bash} and
Heller and Reiner \cite{hel_rei}
classified all representations of the
Klein group (i.e., pairs of commuting
matrices $(A,B)$ satisfying $A^2=B^2=0$
up to simultaneous similarity) over an
algebraically closed field of
characteristic 2. In contrast to this,
Krugljak \cite{kru} showed that if one
could solve the corresponding problem
for groups of type $(p,p)$ with $p>2$,
then one could classify the
representations of {\it any} group over
an algebraically closed field of
characteristic $p$; Heller and Reiner
\cite{hel_rei} showed this for groups
of type (2,2,2).

\section{Representations of quivers}
\label{s2}

Classification problems for systems of
linear mappings may be formulated in
terms of a quiver and its
representations introduced by Gabriel
\cite{gab} (see also \cite{gab_roi}). A
{\it quiver} is a directed graph. Its
{\it representation} ${\cal A}$ over a
field $k$ is given by assigning to each
vertex $v$ a vector space $V_v$ over
$k$ and to each arrow $\alpha:u\to v$ a
linear mapping ${\cal
A}_{\alpha}:V_u\to V_v$ of the
corresponding vector spaces. Two
representations $\cal A$ and ${\cal
A}'$ are {\it isomorphic} if there
exists a system of linear bijections
${\cal S}_v:V_v\to V'_v$ transforming
$\cal A$ to ${\cal A}'$; that is, for
which the diagram
\begin{equation}\label{1.0}
\begin{CD}
 V_u @>{\cal A}_{\alpha}>> V_v\\
 @V{\cal S}_uVV @VV{\cal S}_vV\\
 V'_u @>{\cal A}'_{\alpha}>> V'_v
\end{CD}
\end{equation}
is commutative (${\cal S}_v{\cal
A}_{\alpha}= {\cal A}'_{\alpha}{\cal
S}_u$) for every arrow $\alpha:
u\longrightarrow v$. The {\it direct
sum} of ${\cal A}$ and ${\cal A}'$ is
the representation ${\cal A}\oplus{\cal
A}'$ formed by $V_v\oplus V'_v$ and
${\cal A}_{\lambda}\oplus{\cal
A}'_{\lambda}$.

For example, the problems of
classifying representations of the
quivers \quiv , \!
%TexCad Options
%\grade{\on}
%\emlines{\off}
%\beziermacro{\off}
%\reduce{\on}
%\snapping{\off}
%\quality{2.00}
%\graddiff{0.01}
%\snapasp{1}
%\zoom{1.00}
\unitlength 0.3mm \linethickness{0.4pt}
\begin{picture}(22.12,9.55)(2,1.3)
\put(1.27,4.33){\circle*{0.89}}
\put(21.67,4.33){\circle*{0.89}}
%\bezvec{76}(2.67,4.77)(10.23,9.55)(19.67,5.33)
\put(19.67,5.33){\vector(3,-1){0.2}}
\multiput(2.67,4.77)(0.20,0.11){10}{\line(1,0){0.20}}
\multiput(4.69,5.87)(0.30,0.11){7}{\line(1,0){0.30}}
\multiput(6.78,6.66)(0.54,0.12){4}{\line(1,0){0.54}}
\multiput(8.93,7.14)(1.11,0.08){2}{\line(1,0){1.11}}
\multiput(11.15,7.31)(1.14,-0.07){2}{\line(1,0){1.14}}
\multiput(13.43,7.16)(0.59,-0.11){4}{\line(1,0){0.59}}
\multiput(15.78,6.71)(0.32,-0.11){12}{\line(1,0){0.32}}
%\end
%\bezvec{76}(2.67,3.89)(10.23,-1.00)(19.67,3.33)
\put(19.67,3.33){\vector(3,1){0.2}}
\multiput(2.67,3.89)(0.20,-0.11){10}{\line(1,0){0.20}}
\multiput(4.69,2.76)(0.30,-0.12){7}{\line(1,0){0.30}}
\multiput(6.78,1.95)(0.43,-0.10){5}{\line(1,0){0.43}}
\multiput(8.93,1.47)(1.11,-0.08){2}{\line(1,0){1.11}}
\multiput(11.15,1.30)(1.14,0.07){2}{\line(1,0){1.14}}
\multiput(13.43,1.45)(0.59,0.12){4}{\line(1,0){0.59}}
\multiput(15.78,1.92)(0.32,0.12){12}{\line(1,0){0.32}}
%\end
\end{picture}\,,
and \pair \, are the problems of
classifying linear operators (whose
solution is the Jordan of Frobenius
normal form), pairs of linear mappings
from one space to another (the matrix
pencil problem, solved by Kronecker),
and pairs of linear operators in a
vector space (i.e., pairs of matrices
up to simultaneous similarity).

Furthermore, a representation of the
quiver
\begin{equation}       \label{2.1}
%TexCad Options
%\grade{\on}
%\emlines{\on}
%\beziermacro{\off}
%\reduce{\on}
%\snapping{\on}
%\quality{2.00}
%\graddiff{0.01}
%\snapasp{1}
%\zoom{1.00}
\special{em:linewidth 0.4pt}
\unitlength 0.60mm
\linethickness{0.4pt}
\begin{picture}(139.00,22.00)
(0,17.3)
%\bezier{92}(116.00,11.00)(133.00,15.33)(134.00,10.00)
\emline{116.00}{11.00}{1}{119.51}{11.83}{2}
\emline{119.51}{11.83}{3}{122.64}{12.43}{4}
\emline{122.64}{12.43}{5}{125.39}{12.80}{6}
\emline{125.39}{12.80}{7}{127.76}{12.94}{8}
\emline{127.76}{12.94}{9}{129.75}{12.85}{10}
\emline{129.75}{12.85}{11}{131.37}{12.54}{12}
\emline{131.37}{12.54}{13}{132.61}{12.00}{14}
\emline{132.61}{12.00}{15}{133.47}{11.23}{16}
\emline{133.47}{11.23}{17}{134.00}{10.00}{18}
%\end
%\bezvec{92}(134.00,10.00)(133.00,4.33)(116.00,9.00)
\put(116.00,9.00){\vector(-4,1){0.2}}
\emline{134.00}{10.00}{19}{133.59}{8.89}{20}
\emline{133.59}{8.89}{21}{132.81}{8.02}{22}
\emline{132.81}{8.02}{23}{131.65}{7.40}{24}
\emline{131.65}{7.40}{25}{130.11}{7.02}{26}
\emline{130.11}{7.02}{27}{128.19}{6.89}{28}
\emline{128.19}{6.89}{29}{125.89}{7.00}{30}
\emline{125.89}{7.00}{31}{123.22}{7.36}{32}
\emline{123.22}{7.36}{33}{120.16}{7.96}{34}
\emline{120.16}{7.96}{35}{116.00}{9.00}{36}
%\end
%\bezier{92}(24.00,11.00)(7.00,15.33)(6.00,10.00)
\emline{24.00}{11.00}{37}{20.49}{11.83}{38}
\emline{20.49}{11.83}{39}{17.36}{12.43}{40}
\emline{17.36}{12.43}{41}{14.61}{12.80}{42}
\emline{14.61}{12.80}{43}{12.24}{12.94}{44}
\emline{12.24}{12.94}{45}{10.25}{12.85}{46}
\emline{10.25}{12.85}{47}{8.63}{12.54}{48}
\emline{8.63}{12.54}{49}{7.39}{12.00}{50}
\emline{7.39}{12.00}{51}{6.53}{11.23}{52}
\emline{6.53}{11.23}{53}{6.00}{10.00}{54}
%\end
%\bezvec{92}(6.00,10.00)(7.00,4.33)(24.00,9.00)
\put(24.00,9.00){\vector(4,1){0.2}}
\emline{6.00}{10.00}{55}{6.41}{8.89}{56}
\emline{6.41}{8.89}{57}{7.19}{8.02}{58}
\emline{7.19}{8.02}{59}{8.35}{7.40}{60}
\emline{8.35}{7.40}{61}{9.89}{7.02}{62}
\emline{9.89}{7.02}{63}{11.81}{6.89}{64}
\emline{11.81}{6.89}{65}{14.11}{7.00}{66}
\emline{14.11}{7.00}{67}{16.78}{7.36}{68}
\emline{16.78}{7.36}{69}{19.84}{7.96}{70}
\emline{19.84}{7.96}{71}{24.00}{9.00}{72}
%\end
\put(28.00,10.00){\makebox(0,0)[cc]{1}}
\put(112.00,10.00){\makebox(0,0)[cc]{3}}
%\vector(32.00,7.00)(108.00,7.00)
\put(108.00,7.00){\vector(1,0){0.2}}
\emline{32.00}{7.00}{79}{108.00}{7.00}{80}
%\end
%\vector(32.00,10.00)(108.00,10.00)
\put(108.00,10.00){\vector(1,0){0.2}}
\emline{32.00}{10.00}{73}{108.00}{10.00}{74}
%\end
\put(70.00,34.00){\makebox(0,0)[cc]{2}}
\put(1.00,10.00){\makebox(0,0)[cc]{$\alpha$}}
\put(139.00,10.00){\makebox(0,0)[cc]{$
\zeta$}}
\put(70.00,15.00){\makebox(0,0)[cc]{$\gamma$}}
\put(48.00,25.00){\makebox(0,0)[cc]{$\beta$}}
\put(91.00,25.00){\makebox(0,0)[cc]{$\varepsilon
$}}
%\vector(32.00,12.00)(66.00,29.00)
\put(66.00,29.00){\vector(2,1){0.2}}
\emline{32.00}{12.00}{75}{66.00}{29.00}{76}
%\end
%\vector(74.00,29.00)(108.00,12.00)
\put(108.00,12.00){\vector(2,-1){0.2}}
\emline{74.00}{29.00}{77}{108.00}{12.00}{78}
%\end
\put(70.00,3.00){\makebox(0,0)[cc]{$
\delta$}}
\end{picture}\vspace*{8mm}
\end{equation}
over a field $k$ is a set of linear
mappings
\begin{equation}       \label{2.2}
%TexCad Options
%\grade{\on}
%\emlines{\on}
%\beziermacro{\off}
%\reduce{\on}
%\snapping{\on}
%\quality{2.00}
%\graddiff{0.01}
%\snapasp{1}
%\zoom{1.00}
\special{em:linewidth 0.4pt}
\unitlength 0.70mm
\linethickness{0.4pt}
\begin{picture}(139.00,22.00)
(0,17.3)
%\bezier{92}(116.00,11.00)(133.00,15.33)(134.00,10.00)
\emline{116.00}{11.00}{1}{119.51}{11.83}{2}
\emline{119.51}{11.83}{3}{122.64}{12.43}{4}
\emline{122.64}{12.43}{5}{125.39}{12.80}{6}
\emline{125.39}{12.80}{7}{127.76}{12.94}{8}
\emline{127.76}{12.94}{9}{129.75}{12.85}{10}
\emline{129.75}{12.85}{11}{131.37}{12.54}{12}
\emline{131.37}{12.54}{13}{132.61}{12.00}{14}
\emline{132.61}{12.00}{15}{133.47}{11.23}{16}
\emline{133.47}{11.23}{17}{134.00}{10.00}{18}
%\end
%\bezvec{92}(134.00,10.00)(133.00,4.33)(116.00,9.00)
\put(116.00,9.00){\vector(-4,1){0.2}}
\emline{134.00}{10.00}{19}{133.59}{8.89}{20}
\emline{133.59}{8.89}{21}{132.81}{8.02}{22}
\emline{132.81}{8.02}{23}{131.65}{7.40}{24}
\emline{131.65}{7.40}{25}{130.11}{7.02}{26}
\emline{130.11}{7.02}{27}{128.19}{6.89}{28}
\emline{128.19}{6.89}{29}{125.89}{7.00}{30}
\emline{125.89}{7.00}{31}{123.22}{7.36}{32}
\emline{123.22}{7.36}{33}{120.16}{7.96}{34}
\emline{120.16}{7.96}{35}{116.00}{9.00}{36}
%\end
%\bezier{92}(24.00,11.00)(7.00,15.33)(6.00,10.00)
\emline{24.00}{11.00}{37}{20.49}{11.83}{38}
\emline{20.49}{11.83}{39}{17.36}{12.43}{40}
\emline{17.36}{12.43}{41}{14.61}{12.80}{42}
\emline{14.61}{12.80}{43}{12.24}{12.94}{44}
\emline{12.24}{12.94}{45}{10.25}{12.85}{46}
\emline{10.25}{12.85}{47}{8.63}{12.54}{48}
\emline{8.63}{12.54}{49}{7.39}{12.00}{50}
\emline{7.39}{12.00}{51}{6.53}{11.23}{52}
\emline{6.53}{11.23}{53}{6.00}{10.00}{54}
%\end
%\bezvec{92}(6.00,10.00)(7.00,4.33)(24.00,9.00)
\put(24.00,9.00){\vector(4,1){0.2}}
\emline{6.00}{10.00}{55}{6.41}{8.89}{56}
\emline{6.41}{8.89}{57}{7.19}{8.02}{58}
\emline{7.19}{8.02}{59}{8.35}{7.40}{60}
\emline{8.35}{7.40}{61}{9.89}{7.02}{62}
\emline{9.89}{7.02}{63}{11.81}{6.89}{64}
\emline{11.81}{6.89}{65}{14.11}{7.00}{66}
\emline{14.11}{7.00}{67}{16.78}{7.36}{68}
\emline{16.78}{7.36}{69}{19.84}{7.96}{70}
\emline{19.84}{7.96}{71}{24.00}{9.00}{72}
%\end
\put(28.00,10.00){\makebox(0,0)[cc]{$V_1$}}
\put(112.00,10.00){\makebox(0,0)[cc]{$V_3$}}
%\vector(32.00,7.00)(108.00,7.00)
\put(108.00,7.00){\vector(1,0){0.2}}
\emline{32.00}{7.00}{79}{108.00}{7.00}{80}
%\end
%\vector(32.00,10.00)(108.00,10.00)
\put(108.00,10.00){\vector(1,0){0.2}}
\emline{32.00}{10.00}{73}{108.00}{10.00}{74}
%\end
\put(70.00,34.00){\makebox(0,0)[cc]{$V_2$}}
\put(1.00,10.00){\makebox(0,0)[cc]
{${\cal A}_{\alpha}$}}
\put(139.00,10.00){\makebox(0,0)[cc]
{${\cal A}_{\zeta}$}}
\put(70.00,15.00){\makebox(0,0)[cc]
{${\cal A}_{\gamma}$}}
\put(48.00,25.00){\makebox(0,0)[cc]
{${\cal A}_{\beta}$}}
\put(91.00,25.00){\makebox(0,0)[cc]
{${\cal A}_{\varepsilon}$}}
%\vector(32.00,12.00)(66.00,29.00)
\put(66.00,29.00){\vector(2,1){0.2}}
\emline{32.00}{12.00}{75}{66.00}{29.00}{76}
%\end
%\vector(74.00,29.00)(108.00,12.00)
\put(108.00,12.00){\vector(2,-1){0.2}}
\emline{74.00}{29.00}{77}{108.00}{12.00}{78}
%\end
\put(70.00,3.00){\makebox(0,0)[cc]
{${\cal A}_{\delta}$}}
\end{picture}
\end{equation}\\[6mm]
Let $n_1,n_2,n_3$ be the dimensions of
$V_1,V_2,V_3$; selecting bases in these
spaces, we can give the representation
\eqref{2.2} by the sequence
\begin{multline}\label{2.3}
A=(A_{\alpha},\,A_{\beta},\,
A_{\gamma},\, A_{\delta},\,
A_{\varepsilon},\,A_{\zeta})
\\
 \in k^{n_1\times n_1}\times
 k^{n_2\times n_1}\times
 k^{n_3\times n_1}\times
 k^{n_3\times n_1}\times
 k^{n_3\times n_2}\times
 k^{n_3\times n_3}
\end{multline}
of matrices of linear mappings ${\cal
A}_{\alpha},\,{\cal A}_{\beta},\, {\cal
A}_{\gamma},\, {\cal A}_{\delta},\,
{\cal A}_{\varepsilon},\,{\cal
A}_{\zeta}$. If a sequence of matrices
$A'=(A'_{\alpha},\,A'_{\beta},\dots,
A'_{\zeta})$ gives an isomorphic
representation, then
\begin{equation}\label{2.4}
A'=(S_1A_{\alpha}S_1^{-1},\,
S_2A_{\beta}S_1^{-1},\,
S_3A_{\gamma}S_1^{-1},\,
S_3A_{\delta}S_1^{-1},\,
S_3A_{\varepsilon}S_2^{-1},\, S_3
A_{\zeta}S_3^{-1}),
\end{equation}
where $S_1,S_2,S_3$ are the matrices of
linear bijections ${\cal S}_1,{\cal
S}_2,{\cal S}_3$ (see \eqref{1.0}).
Note that the change of bases in
$V_1,V_2,V_3$ by matrices
$S_1^{-1},S_2^{-1},S_3^{-1}$ also
transforms $A$ to $A'$; that is, $A$
and $A'$ give the same representation
\eqref{2.2} but in different bases.

Therefore, the problem of classifying
representations of the quiver
\eqref{2.1} reduces to the problem of
classifying matrix sequences
\eqref{2.3} up to transformations
\eqref{2.4} with nonsingular matrices
$S_1,S_2,S_3$.

The list of tame quivers and a
classification of their representations
were obtained independently by Donovan
and Freislich \cite{don1} and Nazarova
\cite{naz} (see also \cite[Sect. 11]
{gab_roi}). They proved that a
connected quiver is tame if and only if
it is a subquiver of (or coincides
with) one of the quivers\\[-6mm]
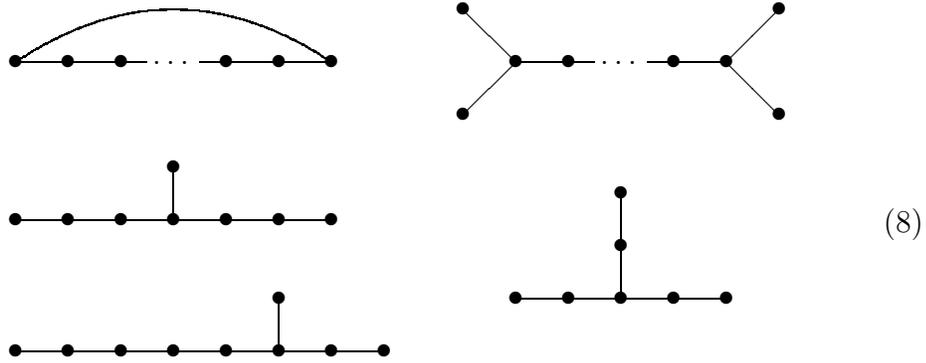
\begin{equation}\label{2.4a}
%TexCad Options
%\grade{\off}
%\emlines{\off}
%\beziermacro{\on}
%\reduce{\off}
%\snapping{\on}
%\quality{2.00}
%\graddiff{0.01}
%\snapasp{5}
%\zoom{1.00}
\unitlength 0.7mm \linethickness{0.4pt}
\begin{picture}(185.00,45.00)(22,0)
\put(40.00,25.00){\makebox(0,0)[cc]{$\bullet$}}
\put(50.00,25.00){\makebox(0,0)[cc]{$\bullet$}}
\put(60.00,25.00){\makebox(0,0)[cc]{$\bullet$}}
\put(40.00,25.00){\line(1,0){25.00}}
\put(70.00,25.00){\makebox(0,0)[cc]{$\dots$}}
\put(80.00,25.00){\makebox(0,0)[cc]{$\bullet$}}
\put(90.00,25.00){\makebox(0,0)[cc]{$\bullet$}}
\put(100.00,25.00){\makebox(0,0)[cc]{$\bullet$}}
\put(100.00,25.00){\line(-1,0){25.00}}
\bezier{288}(40.00,25.00)(70.00,45.00)(100.00,25.00)
\put(125.00,15.00){\makebox(0,0)[cc]{$\bullet$}}
\put(135.00,25.00){\makebox(0,0)[cc]{$\bullet$}}
\put(125.00,35.00){\makebox(0,0)[cc]{$\bullet$}}
\put(145.00,25.00){\makebox(0,0)[cc]{$\bullet$}}
\put(125.00,15.00){\line(1,1){10.00}}
\put(135.00,25.00){\line(1,0){15.00}}
\put(135.00,25.00){\line(-1,1){10.00}}
\put(155.00,25.00){\makebox(0,0)[cc]{$\dots$}}
\put(185.00,15.00){\makebox(0,0)[cc]{$\bullet$}}
\put(175.00,25.00){\makebox(0,0)[cc]{$\bullet$}}
\put(185.00,35.00){\makebox(0,0)[cc]{$\bullet$}}
\put(165.00,25.00){\makebox(0,0)[cc]{$\bullet$}}
\put(185.00,15.00){\line(-1,1){10.00}}
\put(175.00,25.00){\line(-1,0){15.00}}
\put(175.00,25.00){\line(1,1){10.00}}
\put(40.00,-5.00){\makebox(0,0)[cc]{$\bullet$}}
\put(50.00,-5.00){\makebox(0,0)[cc]{$\bullet$}}
\put(60.00,-5.00){\makebox(0,0)[cc]{$\bullet$}}
\put(70.00,-5.00){\makebox(0,0)[cc]{$\bullet$}}
\put(80.00,-5.00){\makebox(0,0)[cc]{$\bullet$}}
\put(90.00,-5.00){\makebox(0,0)[cc]{$\bullet$}}
\put(100.00,-5.00){\makebox(0,0)[cc]{$\bullet$}}
\put(70.00,5.00){\makebox(0,0)[cc]{$\bullet$}}
\put(70.00,5.00){\line(0,-1){10.00}}
\put(40.00,-5.00){\line(1,0){60.00}}
\put(40.00,-30.00){\makebox(0,0)[cc]{$\bullet$}}
\put(50.00,-30.00){\makebox(0,0)[cc]{$\bullet$}}
\put(60.00,-30.00){\makebox(0,0)[cc]{$\bullet$}}
\put(70.00,-30.00){\makebox(0,0)[cc]{$\bullet$}}
\put(80.00,-30.00){\makebox(0,0)[cc]{$\bullet$}}
\put(90.00,-30.00){\makebox(0,0)[cc]{$\bullet$}}
\put(100.00,-30.00){\makebox(0,0)[cc]{$\bullet$}}
\put(40.00,-30.00){\line(1,0){60.00}}
\put(135.00,-20.00){\makebox(0,0)[cc]{$\bullet$}}
\put(145.00,-20.00){\makebox(0,0)[cc]{$\bullet$}}
\put(155.00,-20.00){\makebox(0,0)[cc]{$\bullet$}}
\put(165.00,-20.00){\makebox(0,0)[cc]{$\bullet$}}
\put(175.00,-20.00){\makebox(0,0)[cc]{$\bullet$}}
\put(155.00,-10.00){\makebox(0,0)[cc]{$\bullet$}}
\put(155.00,-10.00){\line(0,-1){10.00}}
\put(155.00,0.00){\makebox(0,0)[cc]{$\bullet$}}
\put(155.00,0.00){\line(0,-1){10.00}}
\put(135.00,-20.00){\line(1,0){40.00}}
\put(110.00,-30.00){\makebox(0,0)[cc]{$\bullet$}}
\put(110.00,-30.00){\line(-1,0){10.00}}
\put(90.00,-20.00){\makebox(0,0)[cc]{$\bullet$}}
\put(90.00,-20.00){\line(0,-1){10.00}}
\end{picture}
\vspace*{22mm}
\end{equation}
with an arbitrary orientation of edges.

As follows from the next theorem, the
problem of classifying quiver
representations has the same complexity
for all wild quivers.

\begin{theorem} \label{t2.1}
The problem of classifying pairs of
matrices up to simultaneous similarity
contains the problem of classifying
representations of an arbitrary quiver.
\end{theorem}

\begin{proof}
We will prove the theorem for
representations of the quiver
\eqref{2.1} since the proof for the
other quivers is analogous. For each
sequence \eqref{2.3}, we construct the
pair of matrices
\begin{equation}       \label{2.5}
(M,N)=\left(\begin{bmatrix}
I_{n_1}&0&0&0\\ 0&2I_{n_2}&0&0\\
0&0&3I_{n_3}&0\\ 0&0&0&4I_{n_3}
\end{bmatrix},
\begin{bmatrix}
A_{\alpha}&0&0&0\\ A_{\beta}&0&0&0\\
A_{\gamma}&0&0&0  \\ A_{\delta}&
A_{\varepsilon}&I_{n_3}& A_{\zeta}
\end{bmatrix}\right).
\end{equation}
Let $(M,N')$ be analogously constructed
from
$$
A'=(A'_{\alpha},\,A'_{\beta},\,
A'_{\gamma},\, A'_{\delta},\,
A'_{\varepsilon},\,A'_{\zeta}),
$$
and let the pairs $(M,N)$ and $(M,N')$
be simultaneously similar:
\begin{equation}\label{2.5'}
  S^{-1}MS=M,\quad S^{-1}NS=N'.
\end{equation}
The equality $MS=SM$ implies
$$
S=S_1\oplus S_2\oplus S_3\oplus S_4.
$$
Equating in $NS=SN'$ the blocks with
indices (4,3) gives $S_3=S_4$. By the
second equality in \eqref{2.5'}, the
pairs $(M,N)$ and $(M,N')$ are
simultaneously similar if and only if
$A'$ is obtained from $A$ by
transformations \eqref{2.4}.
\end{proof}

\section{Representations of posets}
\label{s3}

Many matrix problems may be formulated
in terms of representations of
partially ordered sets introduced by
Nazarova and Roiter \cite{naz_roi}; see
also \cite[Sect. 1.3]{gab_roi}. Let
$\preceq$ be a reflective binary
relation in $T=\{1,\,2,\dots,t\}$. A
{\it representation} of $(T,\preceq)$
is a block matrix
$$A=[A_1\,|\,A_2\,|\dots|\,A_t].$$
Two representations are {\it
isomorphic} if one reduces to the other
by the following transformations:
\begin{itemize}
  \item[(a)] elementary
row-transformations of the whole
matrix;
  \item[(b)] elementary
column-transformations within each
vertical strip;
  \item[(c)] additions of a column of
$A_i$ to a column of $A_j$ if $i\prec
j$.
\end{itemize}
The {\it direct sum} of representations
$A$ and $A'$ is the representation
$$
A\oplus
A'=\left[\begin{tabular}{cc|cc|c|cc}
 $A_1$&0& $A_2$&0& $\dots$ &$A_t$&0\\
 0&$A'_1$& 0&$A'_2$& $\dots$ &0&$A'_t$
\end{tabular}\right].
$$

Without loss of generality, we will
suppose that $(T,\preceq)$ is a
partially ordered set. Indeed, if
$i\prec j$ and $i\succ j$, then we may
join strips $i$ and $j$ to a single
strip with arbitrary
column-transformations within it. If
$i\prec j$ and $j\prec l$, then we may
add a column $a$ of $A_i$ to a column
$c$ of $A_l$ through a column $b$ of
$A_j$:
$$
(a,b,c)\mapsto (a,a+b,c) \mapsto
(a,a+b,a+b+c) \mapsto (a,b,a+b+c)
\mapsto (a,b,a+c).
$$
Hence, we may put $i\prec l$, leaving
the set of admissible transformations
unchanged. Since every partial ordering
relation in a finite set is
supplemented to a linear ordering
relation, we suppose
\begin{equation*}\label{1.1}
i\prec j\ \Longrightarrow \ i<j
\end{equation*}
(that is, every addition between strips
is from left to right).

For instance, every representation of
$(\{1,2,3\},\le)$ reduces to the form
 $$ \left[
 \begin{tabular}{cc|cc|cc}
  $I$&$0$ & $0$&$0$ & $0$&$0$\\
  $0$&$0$ & $I$&$0$ & $0$&$0$\\
  $0$&$0$ & $0$&$0$ & $I$&$0$\\
  $0$&$0$ & $0$&$0$ & $0$&$0$
 \end{tabular}\right].$$

The following theorem is a well-known
corollary of the Krull--Schmidt theorem
\cite[Sect. 1, Theorem 3.6]{bas} for
additive categories (the categories of
representations of quivers and posets
are additive).

\begin{theorem}          \label{t1.1}
Every representation of a quiver or
poset decomposes into a direct sum of
indecomposable representations
uniquely, up to isomorphism of
summands.
\end{theorem}

Nazarova \cite{naz1} proved that a
poset is wild if and only if it
contains a subset from the following
list:
$$
%TexCad Options
%\grade{\off}
%\emlines{\off}
%\beziermacro{\on}
%\reduce{\off}
%\snapping{\on}
%\quality{2.00}
%\graddiff{0.01}
%\snapasp{5}
%\zoom{1.00}
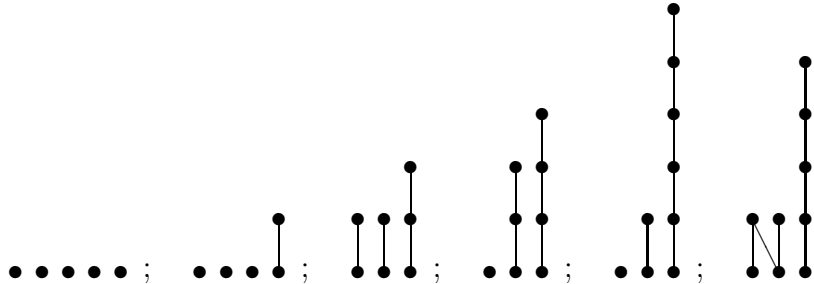
\unitlength 0.7mm \linethickness{0.4pt}
\begin{picture}(155.00,50.00)
\put(0.00,0.00){\makebox(0,0)[cc]{$\bullet$}}
\put(5.00,0.00){\makebox(0,0)[cc]{$\bullet$}}
\put(10.00,0.00){\makebox(0,0)[cc]{$\bullet$}}
\put(15.00,0.00){\makebox(0,0)[cc]{$\bullet$}}
\put(20.00,0.00){\makebox(0,0)[cc]{$\bullet$}}
\put(25.00,0.00){\makebox(0,0)[cc]{;}}
\put(35.00,0.00){\makebox(0,0)[cc]{$\bullet$}}
\put(40.00,0.00){\makebox(0,0)[cc]{$\bullet$}}
\put(45.00,0.00){\makebox(0,0)[cc]{$\bullet$}}
\put(50.00,0.00){\makebox(0,0)[cc]{$\bullet$}}
\put(50.00,10.00){\makebox(0,0)[cc]{$\bullet$}}
\put(50.00,10.00){\line(0,-1){10.00}}
\put(55.00,0.00){\makebox(0,0)[cc]{;}}
\put(65.00,0.00){\makebox(0,0)[cc]{$\bullet$}}
\put(70.00,0.00){\makebox(0,0)[cc]{$\bullet$}}
\put(75.00,0.00){\makebox(0,0)[cc]{$\bullet$}}
\put(75.00,10.00){\makebox(0,0)[cc]{$\bullet$}}
\put(70.00,10.00){\makebox(0,0)[cc]{$\bullet$}}
\put(65.00,10.00){\makebox(0,0)[cc]{$\bullet$}}
\put(65.00,10.00){\line(0,-1){10.00}}
\put(75.00,0.00){\line(0,1){10.00}}
\put(70.00,10.00){\line(0,-1){10.00}}
\put(80.00,0.00){\makebox(0,0)[cc]{;}}
\put(75.00,20.00){\makebox(0,0)[cc]{$\bullet$}}
\put(75.00,20.00){\line(0,-1){10.00}}
\put(90.00,0.00){\makebox(0,0)[cc]{$\bullet$}}
\put(95.00,0.00){\makebox(0,0)[cc]{$\bullet$}}
\put(95.00,10.00){\makebox(0,0)[cc]{$\bullet$}}
\put(95.00,20.00){\makebox(0,0)[cc]{$\bullet$}}
\put(95.00,20.00){\line(0,-1){20.00}}
\put(100.00,0.00){\makebox(0,0)[cc]{$\bullet$}}
\put(100.00,10.00){\makebox(0,0)[cc]{$\bullet$}}
\put(100.00,20.00){\makebox(0,0)[cc]{$\bullet$}}
\put(100.00,30.00){\makebox(0,0)[cc]{$\bullet$}}
\put(100.00,30.00){\line(0,-1){30.00}}
\put(105.00,0.00){\makebox(0,0)[cc]{;}}
\put(115.00,0.00){\makebox(0,0)[cc]{$\bullet$}}
\put(120.00,0.00){\makebox(0,0)[cc]{$\bullet$}}
\put(120.00,10.00){\makebox(0,0)[cc]{$\bullet$}}
\put(120.00,10.00){\line(0,-1){10.00}}
\put(125.00,0.00){\makebox(0,0)[cc]{$\bullet$}}
\put(125.00,10.00){\makebox(0,0)[cc]{$\bullet$}}
\put(125.00,20.00){\makebox(0,0)[cc]{$\bullet$}}
\put(125.00,30.00){\makebox(0,0)[cc]{$\bullet$}}
\put(125.00,40.00){\makebox(0,0)[cc]{$\bullet$}}
\put(125.00,50.00){\makebox(0,0)[cc]{$\bullet$}}
\put(125.00,50.00){\line(0,-1){50.00}}
\put(130.00,0.00){\makebox(0,0)[cc]{;}}
\put(140.00,0.00){\makebox(0,0)[cc]{$\bullet$}}
\put(140.00,10.00){\makebox(0,0)[cc]{$\bullet$}}
\put(145.00,0.00){\makebox(0,0)[cc]{$\bullet$}}
\put(145.00,10.00){\makebox(0,0)[cc]{$\bullet$}}
\put(145.00,10.00){\line(0,-1){10.00}}
\put(145.00,0.00){\line(-1,2){5.00}}
\put(140.00,10.00){\line(0,-1){10.00}}
\put(150.00,0.00){\makebox(0,0)[cc]{$\bullet$}}
\put(150.00,10.00){\makebox(0,0)[cc]{$\bullet$}}
\put(150.00,20.00){\makebox(0,0)[cc]{$\bullet$}}
\put(150.00,30.00){\makebox(0,0)[cc]{$\bullet$}}
\put(150.00,40.00){\makebox(0,0)[cc]{$\bullet$}}
\put(150.00,40.00){\line(0,-1){40.00}}
%\put(155.00,0.00){\makebox(0,0)[cc]{.}}
\end{picture}
$$
($a\prec b$ if $a$ is under $b$ and
they are linked by a line).

As follows from the next theorem and
from the definition of wildness, the
problem of classifying representations
of a poset has the same complexity for
all wild posets.

\begin{theorem} \label{t3.1}
The problem of classifying pairs of
matrices up to simultaneous similarity
contains the problem of classifying
representations of an arbitrary poset.
\end{theorem}

\begin{proof}
{\it Step 1:} Let us prove that the
problem of classifying pairs of
matrices up to simultaneous similarity
contains the problem of classifying
block matrices
\begin{equation}\label{3.1}
  A=\begin{bmatrix}
    A_1 \\ \vdots \\ A_r
  \end{bmatrix},\quad
  A_l=\begin{bmatrix}
    A_{l11}&\dots &A_{l1t} \\
    \hdotsfor{3}\\
   A_{lt1}&\dots &A_{ltt}
  \end{bmatrix},
\end{equation}
up to transformations:
\begin{itemize}
  \item[(i)] arbitrary elementary
transformations within each of $rt$
horizontal strips and each of $t$
vertical strips,

\item[(ii)] additions of columns of
strip $i$ to columns of strip $j$ if
$i<j$,

\item[(iii)] within each $A_l$,
additions of rows of strip $i$ to rows
of strip $j$ if $i>j$
$(i,j\in\{1,\dots,t\})$.
\end{itemize}

We first consider the case $r=1,\ t=3,$
and all $A_{lij}$ of size $1\times 1$.
Then
\begin{equation*}\label{3.2}
 A=\begin{bmatrix}
   a_{11} & a_{12} & a_{13} \\
   a_{21} & a_{22} & a_{23} \\
   a_{31} & a_{32} & a_{33}
 \end{bmatrix}.
\end{equation*}
Basing on $A$, we construct the pair of
matrices
$$
(M,N)= \left( \begin{bmatrix}
    M_1&0 \\ 0&M_2
  \end{bmatrix},\
  \begin{bmatrix}
    0&N_1 \\ 0&0
  \end{bmatrix}\right),
$$
where
$$
M_1=\left[\begin{tabular}{c|cc|ccc}
  1&&&&&   \\ \hline
  &1&&&&   \\
  &1&1&&&  \\ \hline
  &&&1&&   \\
  &&&1&1&  \\
  &&&&1&1
 \end{tabular}\right],\quad
 M_2=\left[\begin{tabular}{c|cc|ccc}
  2&&&&&   \\ \hline
  &2&&&&   \\
  &1&2&&&  \\ \hline
  &&&2&&   \\
  &&&1&2&  \\
  &&&&1&2
 \end{tabular}\right]
$$
(we omit zeros), and
$$
N_1=\left[\begin{tabular}{c|cc|ccc}
 $a_{11}$&$a_{12}$&0&$a_{13}$&0&0
  \\ \hline
 0&0&0&0&0&0 \\
 $a_{21}$&$a_{22}$&0&$a_{23}$&0&0
  \\ \hline
 0&0&0&0&0&0 \\
 0&0&0&0&0&0 \\
 $a_{31}$&$a_{32}$&0&$a_{33}$&0&0
  \end{tabular}\right].
$$
Let $(M,N')$ be analogously constructed
basing on $A'=[a_{ij}']_{i,j=1}^3$, and
let $(M,N)$ be simultaneously similar
to $(M,N')$:
$$
(S^{-1}MS,S^{-1}NS)=(M,N').
$$
Then $MS=SM$, and hence\\
\begin{equation}\label{3.3}
%TexCad Options
%\grade{\off}
%\emlines{\off}
%\beziermacro{\on}
%\reduce{\off}
%\snapping{\on}
%\quality{2.00}
%\graddiff{0.01}
%\snapasp{5}
%\zoom{1.00}
\unitlength 0.5mm \linethickness{0.4pt}
\begin{picture}(30.00,30.00)(-60,0)
\put(-115.00,0.00){\makebox(0,0)[cc]{$S=$}}
\put(-105.00,0.00){\line(1,0){60.00}}
\put(-105.00,20.00){\line(1,0){60.00}}
\put(-95.00,30.00){\line(0,-1){60.00}}
\put(-75.00,-30.00){\line(0,1){60.00}}
\put(-100.00,25.00){\makebox(0,0)[cc]{$\bullet$}}
\put(-90.00,15.00){\makebox(0,0)[cc]{$\bullet$}}
\put(-90.00,5.00){\makebox(0,0)[cc]{$\bullet$}}
\put(-80.00,5.00){\makebox(0,0)[cc]{$\bullet$}}
\put(-70.00,5.00){\makebox(0,0)[cc]{$\bullet$}}
\put(-70.00,-5.00){\makebox(0,0)[cc]{$\bullet$}}
\put(-70.00,-15.00){\makebox(0,0)[cc]{$\bullet$}}
\put(-70.00,-25.00){\makebox(0,0)[cc]{$\bullet$}}
\put(-60.00,-25.00){\makebox(0,0)[cc]{$\bullet$}}
\put(-50.00,-25.00){\makebox(0,0)[cc]{$\bullet$}}
\put(-60.00,-15.00){\makebox(0,0)[cc]{$\bullet$}}
\put(-90.00,-25.00){\makebox(0,0)[cc]{$\bullet$}}
\put(-90.00,15.00){\line(1,-1){10.00}}
\put(-70.00,-5.00){\line(1,-1){20.00}}
\put(-60.00,-25.00){\line(-1,1){10.00}}
\put(-40.00,0.00){\makebox(0,0)[cc]{$\oplus$}}
\put(-35.00,0.00){\line(1,0){60.00}}
\put(-35.00,20.00){\line(1,0){60.00}}
\put(-25.00,30.00){\line(0,-1){60.00}}
\put(-5.00,-30.00){\line(0,1){60.00}}
\put(-30.00,25.00){\makebox(0,0)[cc]{$\bullet$}}
\put(-20.00,15.00){\makebox(0,0)[cc]{$\bullet$}}
\put(-20.00,5.00){\makebox(0,0)[cc]{$\bullet$}}
\put(-10.00,5.00){\makebox(0,0)[cc]{$\bullet$}}
\put(0.00,5.00){\makebox(0,0)[cc]{$\bullet$}}
\put(0.00,-5.00){\makebox(0,0)[cc]{$\bullet$}}
\put(0.00,-15.00){\makebox(0,0)[cc]{$\bullet$}}
\put(0.00,-25.00){\makebox(0,0)[cc]{$\bullet$}}
\put(10.00,-25.00){\makebox(0,0)[cc]{$\bullet$}}
\put(20.00,-25.00){\makebox(0,0)[cc]{$\bullet$}}
\put(10.00,-15.00){\makebox(0,0)[cc]{$\bullet$}}
\put(-20.00,-25.00){\makebox(0,0)[cc]{$\bullet$}}
\put(-20.00,15.00){\line(1,-1){10.00}}
\put(0.00,-5.00){\line(1,-1){20.00}}
\put(10.00,-25.00){\line(-1,1){10.00}}
\put(30.00,0.00){\makebox(0,0)[cc]{,}}
\put(-90.00,25.00){\makebox(0,0)[cc]{$\bullet$}}
\put(-70.00,25.00){\makebox(0,0)[cc]{$\bullet$}}
\put(-70.00,15.00){\makebox(0,0)[cc]{$\bullet$}}
\put(-60.00,5.00){\makebox(0,0)[cc]{$\bullet$}}
\put(-60.00,5.00){\line(-1,1){10.00}}
\put(-100.00,5.00){\makebox(0,0)[cc]{$\alpha$}}
\put(-100.00,-25.00){\makebox(0,0)[cc]{$\beta$}}
\put(-90.00,-15.00){\makebox(0,0)[cc]{$\gamma$}}
\put(-80.00,-25.00){\makebox(0,0)[cc]{$\gamma$}}
\put(-30.00,5.00){\makebox(0,0)[cc]{$\bullet$}}
\put(-30.00,-25.00){\makebox(0,0)[cc]{$\bullet$}}
\put(-20.00,-15.00){\makebox(0,0)[cc]{$\bullet$}}
\put(-10.00,-25.00){\makebox(0,0)[cc]{$\bullet$}}
\put(-10.00,-25.00){\line(-1,1){10.00}}
\put(-20.00,25.00){\makebox(0,0)[cc]{$\delta$}}
\put(0.00,25.00){\makebox(0,0)[cc]{$\varepsilon$}}
\put(0.00,15.00){\makebox(0,0)[cc]{$\zeta$}}
\put(10.00,5.00){\makebox(0,0)[cc]{$\zeta$}}
\linethickness{0.8pt}
\put(-105.00,30.00){\line(0,-1){60.00}}
\put(-105.00,-30.00){\line(1,0){60.00}}
\put(-45.00,-30.00){\line(0,1){60.00}}
\put(-45.00,30.00){\line(-1,0){60.00}}
\put(-35.00,30.00){\line(0,-1){60.00}}
\put(-35.00,-30.00){\line(1,0){60.00}}
\put(25.00,-30.00){\line(0,1){60.00}}
\put(25.00,30.00){\line(-1,0){60.00}}
\end{picture}
\end{equation}
\\[11mm]
where entries linked by lines are
equal. If $S$ is diagonal, then $A'$ is
obtained by multiplying rows and
columns of $A=[a_{ij}]$ by nonzero
scalars. If all diagonal entries of $S$
are 1 and all off-diagonal entries are
0 except for $\alpha$, or $\beta$, or
$\gamma$ (respectively, $\delta$, or
$\varepsilon$, or $\zeta$), then $A'$
is obtained from $A$ by additions of
rows from top to bottom (respectively,
of columns from left to right). The
form of \eqref{3.3} implies that
$(M,N)$ and $(M,N')$ are simultaneously
similar if and only if $A$ reduces to
$A'$ by transformations (i)--(iii).

In common case, $A$ has the form
\eqref{3.1}. Basing on $A$, we
construct $(M,N)$ as follows:
$$
M=M_1\oplus\dots\oplus M_{r+1},\quad
M_l=lI\oplus J_2(lI)\oplus\dots\oplus
J_t(lI),
$$
where
$$
J_i(lI)=\begin{bmatrix}
   lI &&& \\
   I & lI & & \\
    &\ddots &\ddots&\\
    &&I&lI
 \end{bmatrix}
$$
is obtained from the $i\times i$ Jordan
block $J_i(l)$. The matrix $N$ consists
of the blocks $A_{lij}$ (see
\eqref{3.1}) and zeros; each block
$A_{lij}$ is located at the place of
those block of $M$ that is the
intersection of the last horizontal
strip of $J_i(lI)$ and the first
vertical strip of $J_j((r+1)I)$.

If $S$ commutes with $M$, then
$$
S=S_1\oplus\dots\oplus S_{r+1},\quad
S_l=[S_{lij}]_{i,j=1}^t,
$$
where each $S_{lij}$ is of the form
$$
\begin{bmatrix}
   X_1 &&&&&& \\
   X_2 & X_1&&&& & \\
    \vdots&\ddots &\ddots&&&&\\
    X_p&\dots&X_2&X_1&&&
 \end{bmatrix}
 \quad\text{or}\quad
\begin{bmatrix}
   &&&\\ &&&\\
   X_1 &&& \\
   X_2 & X_1&& \\
    \vdots&\ddots &\ddots&\\
    X_p&\dots&X_2&X_1
 \end{bmatrix}.
$$
Therefore, if we restrict ourselves to
those transformations of simultaneous
similarity with $(M,N)$ that preserve
all of its blocks except for $A_{lij}$,
then $A$ reduces by transformations
(i)--(iii).

{}\medskip

 \noindent{\it Step 2}: We prove that the
problem of classifying matrices
\eqref{3.1} up to transformations
(i)--(iii) contains the problem of
classifying representations of each
poset
$$
{\cal P}=(T,\preceq), \quad T=
\{1,\dots,t\}.
$$

Namely, we show that there exists a
block matrix $A^{(r)}=[A_l]_{l=1}^r$ of
the form \eqref{3.1} such that the set
of admissible transformations
(i)--(iii) that preserve
$A_2,\dots,A_r$ produces on its first
strip
\begin{equation}\label{3.4}
[A_{111}\,|\,A_{112}\,| \dots
|\,A_{11t}]
\end{equation}
the matrix problem given by $\cal P$;
we say in this situation that $A$ {\it
simulates} the poset $\cal P$. Of
course, this property does not depend
on the entries of $A_1$, so
$A^{(1)}=[A_1]$ simulates the linearly
ordered set $(T,\le)$.

Every poset $(T,\unlhd)$ is determined
by the set of pairs
$$
G(\unlhd)=\{(i,j)\in T\times T
\,|\,i\unlhd j\}.
$$

We will construct $A_2,\dots,A_r$
sequentially. Suppose
$A^{(m)}=[A_l]_{l=1}^m$ has been
constructed and it simulates a poset
$(T,\unlhd)$ with
$$
G(\preceq)\subset G(\unlhd),\quad
G(\preceq)\ne G(\unlhd).
$$
Let us chose $(a,b)\in
G(\unlhd)\smallsetminus G(\preceq)$ and
construct $A_{m+1}$ so that
$A^{(m+1)}=[A_l]_{l=1}^{m+1}$ simulates
a poset $(T, \sqsubseteq)$, for which
\begin{equation}\label{3.5}
G(\preceq)\subset G(\sqsubseteq)\subset
G(\unlhd),\quad (a,b)\notin
G(\sqsubseteq).
\end{equation}

If we took
\begin{equation}\label{3.5a}
A_{m+1}=\begin{bmatrix}
  0 & \dots &0& I \\
  0 & \dots &I& 0 \\
      \hdotsfor{4}\\
  I & \dots &0& 0
\end{bmatrix},
\end{equation}
then the matrix $A^{(m+1)}$ should
simulate the same poset $(T,\unlhd)$ as
$A^{(m)}$ since every admissible (with
respect to $\unlhd$) transformation
(ii) with columns of $A^{(m+1)}$ spoils
\eqref{3.5a}, but it is restored by
transformations (iii).

Thus we patch up \eqref{3.5a} as
follows. In the set of its horizontal
strips intersecting at $I$ with
vertical strips $a,a+1,\dots,b$, we
make the transposition that gathers at
the top the strips intersecting at $I$
with vertical strips $a,
a_2,\dots,a_l$, where
\begin{equation}\label{3.6}
{\cal A}= \{a,
a_2,\dots,a_l\}=\{i\,|\,a\le i<b,\
a\preceq i\}.
\end{equation}

For instance, if $t=8$, $(a,b)=(3,7)$,
and ${\cal A}=\{3,5,6\}$, then we
obtain
$$
A_{m+1}=\begin{bmatrix}
  0&0&0&0&0&0&0&I\\
  0&0&0&0&0&I&0&0\\
  0&0&0&0&I&0&0&0\\
  0&0&I&0&0&0&0&0\\
  0&0&0&0&0&0&I&0\\
  0&0&0&I&0&0&0&0\\
  0&I&0&0&0&0&0&0\\
  I&0&0&0&0&0&0&0
\end{bmatrix}.
$$

Each addition of a column of strip $i$
to a column of strip $j$, $i<j$, spoils
$A_{m+1}$, but it is restored by
transformations (iii) for all $(i,j)$
except when $(i,j)\in {\cal A}\times
{\cal B}$, where ${\cal
B}=\{a,a+1,\dots,b\} \smallsetminus
{\cal A}$. Hence, the obtained block
matrix $A^{(m+1)}$ simulates the poset
$(T, \sqsubseteq)$ with
$$
G(\sqsubseteq)=G(\unlhd)\smallsetminus
{\cal A}\times {\cal B}.
$$
By \eqref{3.6}, $i\npreceq j$ for all
$(i,j)\in {\cal A}\times{\cal B}$ with
$i<j$. Therefore, the relation
$\sqsubseteq$ satisfies \eqref{3.5}.

We construct $A_2,A_3,\dots$ until
obtain a block matrix
$A^{(r)}=[A_l]_{l=1}^r$ that simulates
the poset ${\cal P}=(T,\preceq)$.
\end{proof}

\section{Spatial matrices and tensors}
\label{s4}

The problem of classifying {\it tensors
of type} $(p,q)$, where $p$ and $q$ are
nonnegative integers, is the problem of
classifying $m\times\dots\times m$
families $[a_{i_1\dots i_{p+q}}]_{
i_1,\dots,i_{p+q}=1}^m$ of elements of
the field up to transformations
\begin{gather*}\label{3.6b}
[a_{i_1\dots i_{p+q}}]_{
i_1,\dots,i_{p+q}=1}^m \longmapsto
[b_{j_1\dots j_{p+q}}]_{
j_1,\dots,j_{p+q}=1}^m,\\ \label{3.6c}
 b_{j_1\dots
j_{p+q}}=\sum_{i_1,\dots,i_{p+q}=1}^m
a_{i_1\dots i_{p+q}}c_{i_1j_1}\dots
c_{i_pj_p}d_{i_{p+1}j_{p+1}}\dots
d_{i_{p+q}j_{p+q}},
\end{gather*}
where $[c_{ij}]=C$ is an arbitrary
nonsingular $m\times m$ matrix and
$[d_{ij}]=C^{\vee}=(C^T)^{-1}$. The
problem of classifying tensors of type
$(2,0)$ or $(1,1)$ is the problem of
classifying bilinear forms or linear
operators. A finite-dimensional algebra
is given by a tensor of type $(2,1)$.

In this section, we study the problem
of classifying three-valent tensors
($p+q=3$). For every $p\in\{0,1,2,3\}$,
we prove that the problem of
classifying tensors of type $(p,3-p)$
contains the problem of classifying
pairs of matrices up to simultaneous
similarity, but is not contained in it.

We start with an investigation of
spatial matrices up to equivalence
since each tensor of type $(p,3-p)$ is
an $m\times m\times m$ spatial matrix
${\mathbb A}$, and admissible
transformations with it are equivalence
transformations \eqref{1.00} given by
matrices \eqref{1.01} of the form
\begin{equation}\label{3.6d}
(R,S,T)=(\underbrace{C, \dots,
C}_{\mbox{$p$-times}}\
,C^{\vee},\dots,C^{\vee}).
\end{equation}

\begin{lemma}\label{l1.1}
For every $m\times n\times q$, the
following three classification problems
are equivalent:
\begin{itemize}
  \item[(i)] The problem of
classifying $m\times n\times q$ spatial
matrices up to equivalence.
  \item[(ii)] The problem of
classifying $q$-tuples of $m\times n$
matrices ${\cal A}=(A_1,\dots,A_q)$ up
to
\begin{itemize}
  \item[(a)]
simultaneous elementary transformations
with $A_1,\dots,A_q$, and
  \item[(b)]
the replacement of $\cal A$ with
\begin{equation}\label{1.2a}
(A_1,\dots,A_q)T=(A_1t_{11}+\dots+A_qt_{q1},
\dots,A_1t_{1q}+\dots+A_qt_{qq}),
\end{equation}
where $T=[t_{ij}]$ is a nonsingular
$q\times q$ matrix.
\end{itemize}
 \item[(iii)] The problem of classifying
spaces of $m\times n$ matrices of
dimension at most $q$ up to
multiplication by a nonsingular matrix
from the left and by a nonsingular
matrix from the right.
\end{itemize}
\end{lemma}

\begin{proof}
An $m\times n\times q$ spatial matrix
${\mathbb A}
=[a_{ijk}]_{i=1}^m{}_{j=1}^n{}_{k=1}^q$
may be given by the $q$-tuple $m\times
n$ matrices
\begin{equation}\label{3.6a}
{\cal A}=(A_1,\dots,A_q), \quad
A_k=[a_{ijk}]_{ij}.
\end{equation}
If ${\mathbb A}$ is determined up to
equivalence, then ${\cal A}$ is
determined up to transformations
(a)--(b); furthermore, the vector space
of $m\times n$ matrices generated by
$A_1,\dots,A_q$ is determined up to
simultaneous multiplications of its
matrices by a nonsingular $m\times m$
matrix from the left and a nonsingular
$n\times n$ matrix from the right.
\end{proof}

Of course, the matrix $T$ from
\eqref{1.2a} is a product of elementary
matrices. Hence, every transformation
(b) is a sequence of elementary
transformations: the transposition of
$A_i$ and $A_j$, the multiplication of
$A_i$ by a nonzero scalar, and the
replacement of $A_i$ by $A_i+bA_j$,
$i\ne j$.

\subsection{Classification of $m\times
n\times 2$ spatial matrices}
\label{s4.1}

For every natural number $r$, we define
two $(r-1)\times r$ matrices
$$
F_r=\begin{bmatrix}
 1&&&0 \\
 &\ddots&&\vdots\\
 &&1&0
 \end{bmatrix},\quad
G_r= \begin{bmatrix}
 0&1&& \\
 \vdots&&\ddots&\\
 0&&&1
 \end{bmatrix}.
$$

\begin{theorem} \label{t4.1}
Over an algebraically closed field,
every pair of $m\times n$ matrices
reduces by transformations (a)--(b) to
a direct sum of the form
\begin{equation}\label{4.1}
\bigoplus_i(F_{r_i},G_{r_i})\oplus
\bigoplus_j (F_{s_j}^T,G_{s_j}^T)\oplus
\bigoplus_{k=1}^q
(I_{l_k},J_{l_k}(\lambda_{k})).
\end{equation}
This sum  is determined uniquely, up to
permutation of summands and up to
linear-fractional transformations of
the sequence of eigenvalues:
\begin{equation}\label{4.1a}
(\lambda_1,\dots,\lambda_q) \longmapsto
\left(\frac{a+b\lambda_1}{c+d\lambda_1},\dots,
\frac{a+b\lambda_q}{c+d\lambda_q}\right),
\end{equation}
where $c+d\lambda_1\ne 0,
\dots,c+d\lambda_q\ne 0$, and $ad-bc\ne
0$.
\end{theorem}

\begin{proof}
Let ${\cal A}=(A_1,A_2)$ be a pair of
$m\times n$ matrices. Using
transformations (a) from Lemma
\ref{l1.1}, we reduce it to
\begin{equation}\label{4.2}
\bigoplus_i(F_{r_i},G_{r_i})\oplus
\bigoplus_j (F_{s_j}^T,G_{s_j}^T)\oplus
\bigoplus_{k=1}^{q_1}
(I_{l_k},J_{l_k}(\lambda_{k}))\oplus
\bigoplus_{k=q_1+1}^{q}
(J_{l_k}(0),I_{l_k})
\end{equation}
(the classification of pencils of
matrices, see \cite{gan}). This sum is
determined uniquely up to permutation
of summands.

We will say that a pair of matrices is
{\it pencil-decomposable} if it reduces
by transformations (a) to a direct sum
of pairs. Each transformation (b) with
$$
T=\begin{bmatrix}
  c&a\\d&b
\end{bmatrix},\quad ad-bc\ne 0
$$
replaces each summand $(P,Q)$ of
\eqref{4.2} with
\begin{equation}\label{4.3}
(P',Q')=(cP+dQ,aP+bQ).
\end{equation}
This pair is pencil-indecomposable
(otherwise, $T^{-1}$ transforms its
direct decomposition to the direct
decomposition of $(P,Q)$, but each
summand of \eqref{4.2} is
pencil-indecomposable). All
indecomposable pairs of $(r-1)\times r$
matrices reduce to $(F_r,G_r)$ by
transformations (a). Hence, if $(P,Q)=
(F_r,G_r)$, then $(P',Q')$ reduces to
$(F_r,G_r)$ too. This proves that every
transformation (b) with the pair ${\cal
A}=(A_1,A_2)$ does not change the
summand $\bigoplus_i(F_{r_i},G_{r_i})$
in the decomposition \eqref{4.2}. The
same holds for the summand $\bigoplus_j
(F_{s_j}^T, G_{s_j}^T)$ too.

If $q_1<q$, then we reduce the pair
\eqref{4.2} to the pair \eqref{4.1}
(with other $\lambda_1,\dots,
\lambda_{q_1}$) as follows. We convert
all summands
$(I_{l_k},J_{l_k}(\lambda_{k}))$ and
$(J_{l_k}(0),I_{l_k})$ to
pencil-indecomposable pairs with
nonsingular first matrices by
transformation \eqref{4.3} with
$c=b=1,\ a=0$, and a nonzero $d$  such
that $d\lambda_1\ne -1,\dots,
d\lambda_{q_1}\ne -1$. Then we reduce
these summands to the form
$(I,J(\lambda))$ by transformations
(a).

Each transformation \eqref{4.3}
converts all summands
$(I_{l_k},J_{l_k}(\lambda_k))$ of
\eqref{4.1} to the pairs of matrices
$(cI_{l_k}+dJ_{l_k}(\lambda_k),
aI_{l_k}+bJ_{l_k}(\lambda_k)),$ which
are simultaneously equivalent to
\begin{equation}\label{4.4}
(I_{l_k},\
(aI_{l_k}+bJ_{l_k}(\lambda_k))\cdot
(cI_{l_k}+dJ_{l_k}(\lambda_k))^{-1}).
\end{equation}
The matrices
$aI_{l_k}+bJ_{l_k}(\lambda_k)$ and
$cI_{l_k}+dJ_{l_k}(\lambda_k)$ are
triangular; their diagonal entries are
$a+b\lambda_k$ and $c+d\lambda_k$.
Hence, the pair of matrices  \eqref{4.4}
is simultaneously equivalent to
$$
\left(I_{l_k},\ J_{l_k}
\left(\frac{a+b\lambda_k}{c+d\lambda_k}
\right)\right),
$$
this gives the transformation
\eqref{4.1a}.
\end{proof}

\subsection{Wildness of tensors and
$m\times n\times 3$ spatial matrices}
\label{s4.2}

\begin{theorem} \label{t4.2}
The problem of classifying $m\times
n\times 3$ spatial matrices up to
equivalence is wild.
\end{theorem}

\begin{proof}
For every pair $(X,Y)$ of $r\times r$
matrices, we construct the triple of
matrices:
$$
(A_1,\,A_2,\,A_3(X,Y)) = (B_1,\, B_2,\,
B_3) \oplus (I_r,I_r,I_r) \oplus
(C_1,\,C_2,\, C_3(X,Y)),
$$
where
$$
(B_1,\,B_2,\,B_3)= \left(
\begin{bmatrix}
I_{6r}&&\\&0&\\&&0
\end{bmatrix},\
\begin{bmatrix}
0&&\\&I_{2r}&\\&&0
\end{bmatrix},\
\begin{bmatrix}
0&&\\&0&\\&&I_{2r}
\end{bmatrix}\right)
$$
and
$$
(C_1,\,C_2,\,C_3(X,Y)) = \left(
I_{4r},\
\begin{bmatrix}
0&&&\\I_r&0&&\\0&I_r&0&\\0&0&I_r&0
\end{bmatrix},\ \begin{bmatrix}
0&&&\\0&0&&\\X&0&0&\\0&Y&0&0
\end{bmatrix}\right).
$$
We will prove that
$(A_1,\,A_2,\,A_3(X,Y))$ reduces to
$(A_1,\,A_2,\,A_3(X',Y'))$ by
transformations (a)--(b) from Lemma
\ref{l1.1} if and only if the pairs of
matrices $(X,Y)$ and $(X',Y')$ are
simultaneously similar.

We write $(M_1,M_2,M_3)\sim
(N_1,N_2,N_3)$ if these triples of
matrices are simultaneously equivalent.

Suppose that $(A_1,\,A_2,\,A_3(X,Y))$
reduces to $(A_1,\,A_2,\,A_3(X',Y'))$
by transformations (a)--(b). Then there
exists a nonsingular $3\times 3$ matrix
$T=[t_{ij}]$ such that
$(A_1,\,A_2,\,A_3(X,Y))T\sim
(A_1,\,A_2,\,A_3(X',Y'))$ (see
\eqref{1.2a}). Hence,
$$
\rank{(A_1t_{1j}+A_2t_{2j}+A_3(X,Y)t_{3j})}=
\begin{cases}
\rank{A_j} &\text{if $j=1$ or $j=2$,}\\
\rank{A_3(X',Y')} &\text{if $j=3$.}
\end{cases}
$$
This implies $t_{ij}=0$ if $i\ne j$
since
\begin{multline*}
\rank{(A_1+A_2\alpha+A_3(X,Y)\beta)}>
\rank{A_1} \\
>\rank{(A_2+A_3(X,Y)\gamma)}
> \rank{A_2}> \rank{A_3(X',Y')}
\end{multline*}
for all $\alpha,\beta,\gamma$ such that
$(\alpha,\beta)\ne (0,0)$ and
$\gamma\ne 0$.

Therefore,
\begin{multline*}
(A_1t_{11},\,A_2t_{22},\,A_3(X,Y)t_{33})
=(B_1t_{11},\, B_2t_{22},\,
B_3t_{33})\\ \oplus
(I_rt_{11},I_rt_{22},I_rt_{33}) \oplus
(C_1t_{11},\,C_2t_{22},\,
C_3(X,Y)t_{33})
\end{multline*}
is simultaneously equivalent to
$$
(A_1,\,A_2,\,A_3(X',Y')) = (B_1,\,
B_2,\, B_3)\oplus (I_r,I_r,I_r) \oplus
(C_1,\,C_2,\, C_3(X',Y')).
$$
They can be considered as isomorphic
representations of the quiver
$
1\underrightarrow{\\[-0.5mm]
\overrightarrow{\rightarrow}}2.
$
By Theorem \ref{t1.1},
\begin{gather}
(I_rt_{11},I_rt_{22},I_rt_{33})\sim
(I_r,I_r,I_r),\label{4.5} \\
(C_1t_{11},\,C_2t_{22},\,
C_3(X,Y)t_{33})\sim(C_1,\,C_2,\,
C_3(X',Y')) \label{4.6}
\end{gather}
since $(B_1t_{11},\, B_2t_{22},\,
B_3t_{33})\sim (B_1,\, B_2,\, B_3)$,
the triples \eqref{4.5} are direct sums
of triples of $1\times 1$ matrices, and
each of the triples \eqref{4.6} cannot
be simultaneously equivalent to a
direct sum containing a triple of
$1\times 1$ matrices. By \eqref{4.5},
$t_{11}=t_{22}=t_{33}$. Then
$(C_1t_{11},\,C_2t_{22},\,
C_3(X,Y)t_{33})\sim (C_1,\,C_2,\,
C_3(X,Y))$, and by \eqref{4.6}
$(C_1,\,C_2,\, C_3(X,Y))\sim
(C_1,\,C_2,\, C_3(X',Y'))$. Since
$C_1=I$, $(C_2,\, C_3(X,Y))$ is
simultaneously similar to $(C_2,\,
C_3(X',Y'))$.

Therefore, there is a nonsingular
matrix $R$ such that
\begin{equation*}\label{4.7}
C_2R=RC_2,\quad C_3(X,Y)R=RC_3(X',Y').
\end{equation*}
By the first equality,
$$R=\begin{bmatrix}
  R_1&&&\\R_2&R_1&&\\R_3&R_2&R_1&\\
  R_4&R_3&R_2&R_1
\end{bmatrix}.
$$
By the second equality, $XR_1=R_1X$ and
$YR_1=R_1Y$.
\end{proof}

In the remaining part of Section
\ref{s4.2}, we prove the following
theorem.

\begin{theorem}\label{t4.2a}
For each $p\in\{0,1,2,3\}$, the problem
of classifying tensors of type
$(p,3-p)$ is wild since it contains the
problem of classifying spatial matrices
up to equivalence.
\end{theorem}

An $m\times n\times q$ spatial matrix
${\mathbb A} =[a_{ijk}]$ may be given
by any of the following sequences of
matrices:
\begin{gather}
{\cal A}^{(1)}=(A_1^{(1)},\dots,
A_m^{(1)}), \quad
A_i^{(1)}=[a_{ijk}]_{jk},\label{4.6a}\\
 {\cal
A}^{(2)}=(A_1^{(2)},\dots, A_n^{(2)}),
\quad
A_j^{(2)}=[a_{ijk}]_{ik},\label{4.6b}\\
 {\cal
A}^{(3)}=(A_1^{(3)},\dots, A_q^{(3)}),
\quad
A_k^{(3)}=[a_{ijk}]_{ij};\label{4.6c}
\end{gather}
the last sequence coincides with
\eqref{3.6a}. They play the same role
in the theory of spatial matrices as
the sequences of rows and columns in
the theory of matrices. If ${\mathbb
A}$ is determined up to equivalence, we
may produce arbitrary elementary
transformations within each of the
sequences \eqref{4.6a}--\eqref{4.6c} by
analogy with transformations (b) from
Lemma \ref{l1.1} for \eqref{3.6a}.
Moreover, two spatial matrices are
equivalent if and only if one reduces
to the other by elementary
transformations within
\eqref{4.6a}--\eqref{4.6c}.

It follows that the triple
\begin{equation}\label{4.6cc}
\rank{\mathbb A}=(r_1,r_2,r_3),\quad
r_i=\rank{\cal A}^{(i)},
\end{equation}
is invariant with respect to
equivalence transformations with
${\mathbb A}$ ($r_i$ is the rank of the
system of matrices ${\cal A}^{(i)}$ in
the vector space of matrices of the
corresponding size). We will say that
${\mathbb A}$ is {\it regular} if
$\rank{\mathbb A}=(m,n,q)$, where
$m\times n\times q$ is the size of
${\mathbb A}$.

Let us make the first $r_1$ matrices in
${\cal A}^{(1)}$ linearly independent
and the others zero by elementary
transformations with ${\mathbb A}$.
Then we reduce the ``new'' ${\cal
A}^{(2)}$ and ${\cal A}^{(3)}$ in the
same way. The obtained spatial matrix
is equivalent to $\mathbb A$ and has
the form ${\mathbb A}'\oplus {\mathbb
O}$, where ${\mathbb A}'$ is a regular
$r_1\times r_2\times r_3$ spatial
matrix and ${\mathbb O}$ is the zero
$(m-r_1)\times (n-r_2) \times (q-r_3)$
spatial matrix. We will call ${\mathbb
A}'$ a {\it regular part} of $\mathbb
A$.

\begin{lemma}\label{l4.1}
Two spatial matrices of the same size
are equivalent if and only if their
regular parts are equivalent.
\end{lemma}

\begin{proof}
Let $\mathbb A$ and $\mathbb B$ be
$m\times n\times q$ spatial matrices.
Without loss of generality, we will
assume that
\begin{equation}\label{4.6d}
{\mathbb A}={\mathbb A}'\oplus {\mathbb
O}, \quad {\mathbb B}={\mathbb
B}'\oplus {\mathbb O},
\end{equation}
where ${\mathbb A}'$ and ${\mathbb B}'$
are their regular parts.

{\it Necessity.} Suppose that $\mathbb
A$ and $\mathbb B$ are equivalent, and
their equivalence is given by matrices
$R,S$, and $T$ (see \eqref{1.01}). Then
${\mathbb A}'$ and ${\mathbb B}'$ have
the same size $r_1\times r_2\times
r_3$, where $(r_1, r_2,
r_3)=\rank{\mathbb A}$. Following
\eqref{3.6a}, we will give $\mathbb A$
and $\mathbb B$ by the sequences ${\cal
A}=(A_1,\dots,A_q)$ and ${\cal
B}=(B_1,\dots,B_q)$. Put
\begin{equation}\label{4.6e}
(C_1,\dots, C_q)=(A_1,\dots, A_q)T,
\end{equation}
then
\begin{equation}\label{4.6f}
(R^TC_1S,\dots, R^TC_qS)=
(B_1,\dots,B_q)
\end{equation}
by analogy with transformations
(a)--(b) from Lemma \ref{l1.1}.

Let us partition $R,S$, and $T$ into
blocks $R=[R_{ij}]_{i,j=1}^2$,
$S=[S_{ij}]_{i,j=1}^2$, and
$T=[T_{ij}]_{i,j=1}^2$ in accordance
with the decompositions \eqref{4.6d},
where $R_{11},S_{11}$, and $T_{11}$
have sizes $r_1\times r_1$, $r_2\times
r_2$, and $r_3\times r_3$. If
$T_{21}\ne 0$, then $C_l\ne 0$ for a
certain $l>r_3$ since $A_1,\dots,
A_{r_3}$ are linearly independent (see
\eqref{4.6e} and \eqref{4.6d}). By
\eqref{4.6f}, $B_l\ne 0$; a
contradiction.

Hence, $T_{21}=0$; analogously
$R_{21}=0$ and $S_{21}=0$. It follows
that $R_{11},S_{11}$, and $T_{11}$ are
nonsingular and produce an equivalence
of ${\mathbb A}'$ to ${\mathbb B}'$.

{\it Sufficiency.} Suppose that
${\mathbb A}'$ and ${\mathbb B}'$ are
equivalent and their equivalence is
given by matrices $R,S$, and $T$. Then
the matrices $R\oplus I_{m-r_1},\,
S\oplus I_{n-r_2}$, and $T\oplus
I_{q-r_3}$ produce an equivalence of
${\mathbb A}$ to ${\mathbb B}$.
\end{proof}

\begin{proof}[Proof of Theorem \ref{t4.2a}]
For an $m\times n\times q$ spatial
matrix $\mathbb A$, we construct the
spatial block matrix
\begin{equation}\label{4.8}
{\mathbb H}({\mathbb A})= [{\mathbb
H}_{ijk}]_{i,j,k=1}^3,\quad {\mathbb
H}_{ijk}=
  \begin{cases}
   {\mathbb A} & \text{if $(i,j,k)=(1,2,3)$}, \\
   {\mathbb O} & \text{otherwise},
  \end{cases}
\end{equation}
where the diagonal blocks ${\mathbb
H}_{111},\,{\mathbb H}_{222},$ and
${\mathbb H}_{333}$ have sizes $m\times
m\times m,\, n\times n\times n,$ and
$q\times q\times q.$

Let $\mathbb B$ be another $m\times
n\times q$ spatial matrix. Then
$\mathbb A$ is equivalent to $\mathbb
B$ if and only if ${\mathbb H}({\mathbb
A})$ and ${\mathbb H}({\mathbb B})$
determine the same tensor of type $(p,
3-p)$. Indeed, if matrices
$Q_1,Q_2,Q_3$ give an equivalence of
$\mathbb A$ to $\mathbb B$, then
${\mathbb H}({\mathbb A})$ reduces to
${\mathbb H}({\mathbb B})$ by
equivalence transformations
\eqref{1.00} satisfying \eqref{3.6d},
where $C$ is
$$
Q_1^{\vee}\oplus Q_2^{\vee}\oplus
Q_3^{\vee},\
 Q_1\oplus Q_2^{\vee}\oplus
Q_3^{\vee},\
 Q_1\oplus Q_2\oplus Q_3^{\vee},\
 \text{ or }\
 Q_1\oplus Q_2\oplus Q_3
$$
if, respectively, $p$ is $0,1,2$, or 3.
Conversely, if ${\mathbb H}({\mathbb
A})$ is reduced to ${\mathbb
H}({\mathbb B})$ by transformations
\eqref{1.00}, then they are equivalent.
Since their regular parts are regular
parts of ${\mathbb A}$ and ${\mathbb
B}$ too, ${\mathbb A}$ and ${\mathbb
B}$ are equivalent by Lemma \ref{l4.1}.

We have proved that the problem of
classifying tensors of type $(p,3-p)$
contains the problem of classifying
spatial matrices up to equivalence. By
Theorems \ref{t4.2}, the first problem
is wild.
\end{proof}

\subsection{Spatial matrices and tensors
are ``very wild''} \label{s4.3}

\begin{theorem} \label{t4.3}
The problem of classifying pairs of
matrices up to simultaneous similarity
does not contain both
\begin{itemize}
  \item[(i)] the problem of
classifying $m\times n\times 2$ spatial
matrices up to equivalence, and
  \item[(ii)] the problem of
classifying tensors of type $(p,3-p)$
for each $p\in\{0,1,2,3\}$.
\end{itemize}
\end{theorem}

\begin{proof}
(i) To the contrary, suppose there
exists a pair ${\cal T}(x_1,x_2)$ of
matrices, whose entries are
noncommutative polynomials in
$x_1,x_2$, such that a pair
$A=(A_1,A_2)$ of $m\times n$ matrices
reduces to $A'=(A'_1,A'_2)$ by
transformations (a)--(b) from Lemma
\ref{l1.1} if and only if ${\cal T}(A)$
is simultaneously similar to ${\cal
T}(A')$.

Put
$$
A=\left(\begin{bmatrix}
  1&&\\&0&\\&&1
\end{bmatrix},\
\begin{bmatrix}
  0&&\\&1&\\&&0
\end{bmatrix}\right),\quad
A'=\left(\begin{bmatrix}
  1&&\\&0&\\&&1
\end{bmatrix},\
\begin{bmatrix}
  0&&\\&1&\\&&1
\end{bmatrix}\right).
$$

Since the pair of $1\times 1$ matrices
$([1],[0])$ reduces to $([1],[1])$ by
transformations (b), the pair  ${\cal
T}([1],[0])$ is simultaneously similar
to ${\cal T}([1],[1])$. Therefore,
${\cal T}(A)={\cal T}([1],[0])\oplus
{\cal T}([0],[1])\oplus {\cal
T}([1],[0])$ is simultaneously similar
to ${\cal T}(A')={\cal
T}([1],[0])\oplus {\cal
T}([0],[1])\oplus {\cal T}([1],[1])$,
and hence $A$ reduces to $A'$ by
transformations (a)--(b). By definition
of transformations (a)--(b), there
exist $\alpha,\beta,\gamma,\delta$ such
that $\alpha\delta-\beta\gamma\ne 0$
and
$$
A''=\left(\begin{bmatrix}
  \alpha&&\\&\beta&\\&&\alpha
\end{bmatrix},\
\begin{bmatrix}
  \gamma&&\\&\delta&\\&&\gamma
\end{bmatrix}\right)
$$
is simultaneously equivalent to $A'$.
Equating the ranks of matrices in $A''$
and $A'$ gives $\beta=\delta=0$,
contrary to
$\alpha\delta-\beta\gamma\ne 0$.

(ii) Suppose the problem of classifying
pairs of matrices up to simultaneous
similarity contains the problem of
classifying tensors of type $(p,3-p)$.
Then, by Theorem \ref{t4.2a}, the first
problem contains the problem of
classifying spatial matrices up to
equivalence, contrary to (i).
\end{proof}

{\it The authors wish to thank
Professors Yuri\u{\i} Drozd and Leiba
Rodman for stimulating discussions.}

\end{document}